\RequirePackage{fix-cm}
\documentclass[smallextended]{svjour3}       
\smartqed  
\usepackage{amsmath}
\usepackage{amssymb}
\usepackage{graphicx}
\usepackage{mathrsfs}
\usepackage{caption}
\usepackage{float}
\usepackage{multicol}
\usepackage[ruled,vlined]{algorithm2e}
\begin{document}

\title{A Two-Phase Quasi-Newton Method for Optimization Problem
}


\author{Suvra Kanti Chakraborty         \and
        Geetanjali Panda 
}


\institute{S. K. Chakraborty \at
               Department of Mathematics, Indian Institute of Technology Kharagpur, India, 721302.\\
              \email{suvrakanti@maths.iitkgp.ernet.in}           
           \and
           G. Panda \at
               Department of Mathematics, Indian Institute of Technology Kharagpur, India, 721302.
}

\date{Received: date / Accepted: date}

\maketitle

\begin{abstract}
In this paper, a two-phase quasi-Newton scheme is proposed for solving an unconstrained optimization problem. The global convergence property of the scheme is provided under mild assumptions. The super linear rate of the scheme is also proved in the vicinity of the solution. The advantages of the proposed scheme over the traditional scheme are justified with numerical table and graphical illustrations.
\end{abstract}

\section{Introduction}
The basic idea behind the quasi-Newton methods for unconstrained optimization problem lies on updating the hessian approximation in a computationally cheap way, that ensures the secant condition. DFP and BFGS methods are popular quasi-Newton schemes. DFP method is sensitive to inaccurate line searches, where as BFGS scheme is widely used for its robustness and self correcting properties. quasi-Newton methods have super linear convergence property under exact or inexact line searches. One may see ( \cite{biglari2015limited}, \cite{byrd2016stochastic},  \cite{chen2014global}, \cite{jensen2016approach}, \cite{jie2007quasi}, \cite{tapia2015averaging},\cite{xiao2013global}, \cite{yuan2013bfgs})
for the recent advances in this direction.
Consider a general minimization problem
\begin{equation}
(P)~~Min_{x\in \mathbb{R}^{n}} ~ f(x), \nonumber
\end{equation}
where $f : R^n \rightarrow R$ is a twice continuously differentiable function. Solution of (P) satisfies the system of equations $\nabla f(x)=0$. Several advanced iterative processes exist so far to solve a system of nonlinear equations $F(x),~ F: \mathbb{R}^{n} \rightarrow \mathbb{R}^m$ (\cite{Lukšan2017}, \cite{sharma2013efficient}, \cite{soleymani2014multi}). The two-phase iterative scheme for the same, proposed by  (\cite{babajee2006analysis})
takes the following form.
$$x_{k+1}= x_{k}-[\frac{1}{2} (J(x_{k}) + J(z_{k}))]^{-1} F(x_{k}),$$
$$z_{k}=x_{k}- J(x_{k})^{-1} F(x_{k}),$$
where $J(x_{k})$ is the Jacobian of $F(x)$, at the $k$-th iteration of the iterative scheme. These two-phase schemes have been further studied in unconstrained and constrained optimization problems(\cite{chakraborty2014higher}, \cite{chakraborty2016Two}).
\\\\
In this paper, we propose a variant of quasi-Newton scheme for unconstrained optimization problem, which involves two phases and based on the basic idea of above iterative process. The new scheme  guarantees the super linear global convergence behavior. In numerical illustrations, it is observed that the number of iterations of the proposed scheme are significantly less than the standard BFGS scheme in comparable execution time.\\\\
This concept is organized in following sections.
Two-phase quasi-Newton scheme is proposed in Section \ref{proposed_scheme}, convergence analysis of the scheme is discussed in Section \ref{scheme_convergence}, followed by Numerical illustrations in Section \ref{numerical}.
\section{Proposing two-phase quasi-Newton scheme} \label{proposed_scheme}
In a general line search method for $(P)$, a descent sequence $\{x_{k}\}$ is generated as $x_{k+1}=x_{k}+\alpha_{k}p_{k},$ where $p_{k}$ is the descent direction at $x_k$ and the positive scalar $\alpha_{k}$ is the step length at $x_k$ along the direction $p_{k}$. The condition $p_{k}^ {T} \nabla f_{k}<0$ guarantees that the function $f$ can be reduced along the direction $p_{k}$. The search direction may take the form $p_{k}=-B_{k}^{-1} \nabla f_{k}$, where $B_{k}$ is a symmetric positive definite matrix of order $n \times n$. The most popular quasi-Newton (BFGS) scheme generates a sequence of positive definite  hessian approximation $B_k$ at the current iterate $x_k$ as
\begin{equation}
B_{k+1}=B_{k}-\frac{B_{k} s_{k} s_{k}^{T} B_{k}}{s_{k}^{T}B_{k}s_{k}}+\frac{ y_{k}y_{k}^{T} }{y_{k}^{T}s_{k}}, \nonumber
\end{equation}
where
$s_{k}=x_{k+1}-x_{k},~~y_{k}= \nabla f_{k+1} - \nabla {f_{k}},~~ x_{k+1}=x_{k}-\alpha_{k} B_{k}^{-1} \nabla f_{k},~f_{k}=f(x_{k})$ and $\alpha_{k}$ is computed from a line search procedure. The sequence $\{x_k\}$ shows global super linear convergence behavior under some standard assumptions.
The proposed modified scheme is as follows. \\
Consider a positive definite matrix $B_0$ at initial stage $k=0$ and
\begin{equation}
B_{\mathbf{\bar{k}}}=B_{k}-\frac{B_{k} s_{k} s_{k}^{T} B_{k}}{s_{k}^{T}B_{k}s_{k}}+\frac{ y_{k}y_{k}^{T} }{y_{k}^{T}s_{k}}. \nonumber
\end{equation}
 $B_{k}$ is upgraded as
$$
B_{k+1}=\lambda B_{k}+ (1- \lambda) B_{\mathbf{\bar{k}}}, ~\mbox{for}~ \lambda \in (0,1),
$$
where $
s_{k}=x_{\mathbf{\bar{k}}}-x_{k},~~
 y_{k}= \nabla f_{\mathbf{\bar{k}}} - \nabla {f_{k}},
 ~~~~
 x_{\bar{k}}=x_{k}- \alpha_{\mathbf{\bar{k}}} B_{k}^{-1} \nabla f_{k}.
 $
 Upgrade $x_k$ as
\begin{equation} \label{sch}
x_{k+1}=x_{k}-\alpha_{k} B_{k+1}^{-1} \nabla f_{k},
\end{equation}
where $\alpha_{\mathbf{\bar{k}}}$ and $\alpha_{k}$ are computed by inexact Wolfe search.\\ \\

In this proposed method we modify the search direction in two phases as follows. The step length $\alpha_{k}$ and $\alpha_{\mathbf{\bar{k}}}$ are chosen by Wolfe search method.
\begin{equation}\label{AAAA}
x_{\mathbf{\bar{k}}}=x_{k}+\alpha_{\mathbf{\bar{k}}} p_{\mathbf{\bar{k}}}, ~\mbox{where} ~~p_{\mathbf{\bar{k}}}=-B_{k}^{-1}\nabla f_{k},
\end{equation}
and
\begin{equation} \label{BBBB}
x_{k+1}=x_{k}+ \alpha_{k} p_{k},
~\mbox{where}~~
p_{k}=-B_{k+1}^{-1} \nabla f_{k}.
 \end{equation}
If $B_{k}$ is positive definite then this new matrix $B_{k+1}$ is positive definite being the arithmetic mean of $B_{k}$ and the standard BFGS updation of it. $B_{k+1} $ is used at two consecutive stages:  at the $k^{th}$ stage to determine $x_{k+1}$, as well as at $k$+1$^{th}$ stage to determine $x_{\mathbf{\overline{k+1}}}$. In the vicinity of the solution of $(P)$, the parameters $\alpha_{k}$ and $\alpha_{\mathbf{\bar{k}}}$  are chosen as unit length for $k$ greater than sufficiently large $k_{0} \in \mathbb{N}$. This fact will be used in next section to prove the convergence of the scheme.
\section{Convergence of the proposed scheme} \label{scheme_convergence}
\begin{theorem}\label{convergence_criterion}
Suppose $f : \mathbb{R}^{n} \rightarrow \mathbb{R}$ is twice continuously differentiable and $x_{k}$ is updated by (\ref{AAAA}) and (\ref{BBBB}). Assume that the sequence $\{x_{k}\}$ generated by (\ref{sch}) satisfies $\{x_{k}\} \rightarrow x^{*}$, $\nabla f(x^{*})=0$ and $\nabla ^{2} f(x^{*})>0$. Then the sequence $\{x_{k}\}$ converges to $x^{*}$ super linearly if
\begin{equation} \label{condition}
\lim_{k \rightarrow \infty} \frac{\| (B_{k} - \nabla^{2} f^{*})p_{\mathbf{\bar{k}}} \| }{ \|  p_{\mathbf{\bar{k}}} \| }=0.
\end{equation}
\end{theorem}
\begin{proof}
Denote $f(x^{*})=f^{*}$. Newton direction at $x_k$ is $p_k^N=-\nabla^{2} f_{k}^{-1} \nabla f_{k}$. Then from ( \ref{AAAA}),
\begin{align}
p_{\mathbf{\bar{k}}}-p_{k}^{N} &= \nabla^{2} f_{k}^{-1} \big( \nabla^{2} f_{k} p_{\mathbf{\bar{k}}}+ \nabla f_{k} \big) =
 \nabla^{2} f_{k}^{-1}  (\nabla^{2} f_{k}- B_{k}) p_{\mathbf{\bar{k}}}. \label{ZZZ}
\end{align}
 Since the limiting hessian $\nabla^{2} f^{*}$ is positive definite, so $\| \nabla^{2} f_{k}^{-1} \| $ is bounded above for $x_{k}$ sufficiently close to $x^{*}$, and hence from (\ref{ZZZ}),
\begin{align}
\| p_{\mathbf{\bar{k}}}-p_{k}^{N} \| & \leq M~ ( \| (\nabla^{2} f_{k}- B_{k}) p_{\mathbf{\bar{k}}} \| )~~~\mbox{for some positive real number M}   \nonumber \\
& \leq M~ \big( \| (B_{k}-\nabla^{2} f^{*} ) p_{\mathbf{\bar{k}}} \| + \| (\nabla^{2} f_{k}-\nabla^{2} f^{*} )\| \| p_{\mathbf{\bar{k}}} \| \big). \nonumber
 \end{align}
Dividing both side by $\| p_{\mathbf{\bar{k}}} \|$, letting  $k \rightarrow \infty $, $x_{k} \rightarrow x^{*}$ and using (\ref{condition}), we have
 \begin{align}\label{XXXX}
&\frac{\| p_{\mathbf{\bar{k}}}-p_{k}^{N} \|}{ \| p_{\mathbf{\bar{k}}} \|} \rightarrow 0.
\end{align}
Note that for sufficiently large $k\in \mathbb{N}$, $\| p_{k} - p_{\mathbf{\bar{k}}} \| \rightarrow 0$, since
\begin{align} \label{PPP}
\| p_{k} - p_{\mathbf{\bar{k}}} \| = \| x_{k+1} - x_{\mathbf{\bar{k}}} \| \leq \| x_{k+1} - x^{*} \| + \| x_{\mathbf{\bar{k}}} - x^{*} \| \leq \frac{\epsilon}{2} + \frac{\epsilon}{2} = \epsilon,
\end{align}
where $\epsilon$ is arbitrarily small real number.
 For some $N_{2} \in \mathbb{N}$,
$$\| p_{\mathbf{\bar{k}}} \|=\| x_{\bar{k}}-x_{k} \| \leq \| x_{k} -x^{*} \| + \| x_{\bar{k}} -x^{*} \| \leq 2\| x_{k} - x^{*} \|~ \mbox{for all}~ k>N_{2}, $$
which implies
$
\| p_{\mathbf{\bar{k}}} \| =   \mbox{O} ( \| x_{k} - x^{*} \| ).
$
Using (\ref{XXXX}) and (\ref{PPP}), we have
 \begin{align*}\label{proof}
\| x_{k}+p_{k}-x^{*} \| & \leq \| x_{k} + p_{k}^{N}-x^{*} \| + \| p_{\mathbf{\bar{k}}} - p_{k}^{N}\| + \|p_{k} -p_{\mathbf{\bar{k}}}\| \nonumber \\
&\leq \mbox{O} (\| x_{k}-x^{*} \| ^{2}) +  \mbox{o}(\| p_{\mathbf{\bar{k}}} \|)
 \leq \mbox{o} ( \| x_{k} -x^{*} \|)
\end{align*}
%
This shows super linear convergence behavior of the sequence $\{x_k\}$.
\end{proof}
\subsection{Convergence analysis} \label{Convergence}
By Taylor series expansion, $ \nabla f_{\mathbf{\bar{k}}} = \nabla {f_{k}} + \int_{0}^{1} \nabla^{2} f(x+ts_{k})s_{k}~ dt $. Denote $\bar{G}_{k}$ as the average hessian $\int_{0}^{1} \nabla^{2} f(x+ts_{k})~ dt $ . Then
\begin{equation} \label{yGbar}
y_{k}= \nabla f_{\mathbf{\bar{k}}} - \nabla {f_{k}}=\bar{G}_{k}s_{k}.
\end{equation}
\begin{theorem} The sequence $\{ x_k\}$ generated by (\ref{sch}) possesses global convergence property under the following assumptions.
\begin{enumerate}
\item For some $x_0 \in \mathbb{R}^{n}$, the level set $\mathcal{{L}}=\{ x \in \mathbb{R}^{n} | f(x) \leq f(x_{0}) \}$ is convex and there exist positive constants $m$ and $M$ such that $m \| x \|^{2} \leq x^{T} G_{k} x \leq M \| x \|^{2} ~~\forall x\in \mathbb{R}^n. $
\item $( p_{k}-  p_{\mathbf{\bar{k}}})^{T} \nabla f(x_{\mathbf{\bar{k}}})<0 $, where $G_{k}=\nabla^2 f(x_{k})$.
\item Step length is found by Wolfe's inexact line search method.
\end{enumerate}
\end{theorem}
\begin{proof}
Denote
$m_{k} = \frac{y_{k}^{T} s_{k}}{s_{k}^{T} s_{k}}, ~~M_{k} = \frac{y_{k}^{T}y_{k}}{y_{k}^{T}s_{k}}.$
From Assumption 1,
\begin{equation} \label{mmk}
m_{k}=\frac{y_{k}^{T} s_{k}}{s_{k}^{T} s_{k}}=\frac{s_{k}^{T} \bar{G}_{k} s_{k}}{s_{k}^{T}s_{k}} \geq m
~~\mbox{and}~~
M_{k}=\frac{y_{k}^{T}y_{k}}{y_{k}^{T}s_{k}}=\frac{s_{k}^{T} \bar{G}_{k}^{2} s_{k}}{s_{k}^{T}\bar{G}_{k}s_{k}}=\frac{z_{k}^{T}\bar{G}_{k}z_{k}}{z_{k}^{T}z_{k}} \leq M.
\end{equation}
We have
\begin{align}\label{compact_form}
B_{k+1}=\lambda B_{k} + (1- \lambda) B_{\mathbf{\bar{k}}} =
B_{k} - (1- \lambda) \frac{B_{k} s_{k} s^{T}_{k} B_{k}}{s_{k}^{T}B_{k} s_{k}}+ (1-\lambda) \frac{y_{k} y_{k}^{T}}{y_{k}^{T} s_{k}}.
\end{align}
Then
$
\mbox{trace}(B_{k+1})=\mbox{trace}(B_{k})-(1-\lambda) \frac{\| B_{k} s_{k} \|^{2}}{s_{k}^{T}B_{k}s_{k}}+ (1-\lambda) \frac{\| y_{k} \|^{2} }{y_{k}^{T}s_{k}}.
$
Let $\theta_{k}$ be the angle between $s_{k}$ and $B_{k}s_{k}$. Then $\cos \theta_{k} = \frac{s_{k}^{T} B_{k} s_{k}}{ \| B_{k} s_{k} \| \|s_{k} \| }$.
So
\begin{equation*} 
(1-\lambda) \frac{\| B_{k} s_{k}\|^{2}}{s_{k}^{T}B_{k}s_{k}} = \frac{\| B_{k} s_{k} \|^{2} \|s_{k}\|^{2}}{(s_{k}^{T}B_{k}s_{k})^{2}}. (1-\lambda)\frac{s_{k}^{T}B_{k}s_{k}}{\|s_{k}\|^{2}}=\frac{q_{k}}{\cos^{2} \theta_{k}},
\end{equation*}
where $
q_{k}=(1-\lambda) \frac{s_{k}^{T}B_{k}s_{k}}{s_{k}^{T}s_{k}}.$
Hence \begin{equation} \label{tracev2}
\mbox{trace}(B_{k+1})=\mbox{trace}(B_{k}) + (1-\lambda) M_{k}-\frac{q_{k}}{\cos^{2}\theta_{k}}
\end{equation}
For any $a, b, u, v \in \mathbb{R}^n,$ it is true that $\mbox{det}(I + ab^{T}+uv^{T})=(1+b^{T}a)(1+v^{T}u)-(a^{T}v)(b^{T}u).$ \\\\
Hence taking $a=s_k, b=-(1-\lambda) \frac{ B_ks_k}{s_k^T B_k s_k}, u= B_k^{-1}y_k, v=(1-\lambda) \frac{y_k}{y_k^T s_k}$, we get from (\ref{compact_form}),
\begin{align} \label{detv}
\mbox{det}(B_{k+1})&=\mbox{det}(B_{k}) \Big[ \Big(1- (1-\lambda) \frac{(B_{k}s_{k})^{T}s_{k}}{s_{k}^{T}B_{k}s_{k}}\Big)\Big(1+ (1-\lambda) \frac{y_{k}^{T}B_{k}^{-1}y_{k}}{y_{k}^{T}s_{k}}\Big) + (1-\lambda)^2 \frac{s_{k}^{T}y_{k}}{s_{k}^{T}B_{k}s_{k}}\frac{s_{k}^{T}y_{k}}{y_{k}^{T}s_{k}}\Big] \nonumber \\
&=\mbox{det}(B_{k})\Big[\lambda + \lambda (1-\lambda) \frac{y_{k}^{T}B_{k}^{-1}y_{k}}{y_{k}^{T}s_{k}} + (1-\lambda)^2 \frac{s_{k}^{T}y_{k}}{s_{k}^{T}B_{k}s_{k}}\frac{s_{k}^{T}y_{k}}{y_{k}^{T}s_{k}}\Big] \nonumber \\
&>\mbox{det}(B_{k})\Big[ (1 - \lambda)^2 \frac{s_{k}^{T}y_{k}}{s_{k}^{T}B_{k}s_{k}}\Big]=\mbox{det}(B_{k}) (1-\lambda)^3 \frac{m_{k}}{q_{k}},
\end{align}
Define a function $\psi : \mathbb{R}^{n\times n} \rightarrow \mathbb{R}$ as $\psi(B)= \mbox{trace}(B)-\ln(\mbox{det}(B))$ where $B$ is a positive definite matrix. Since the eigenvalues $\lambda_{i}$'s of $B$ are positive so $$\psi(B)=\mbox{trace}(B)-\ln(\mbox{det}(B))=\sum_{i=1}^{n} \Big( \lambda_{i}-\ln (\lambda_{i}) \Big) >0.$$
From (\ref{tracev2}) and (\ref{detv}),
\begin{align}\label{psi1}
\psi(B_{k+1})&=\mbox{trace}(B_{k+1})-\ln \Big(\mbox{det}(B_{k+1}) \Big) \nonumber \\
&= \mbox{trace}(B_{k}) + (1- \lambda) M_{k}-\frac{q_{k}}{\cos^{2}\theta_{k}} -\ln\Big(\det(B_{k+1})\Big) \nonumber \\
& < \mbox{trace}(B_{k}) + (1- \lambda) M_{k}-\frac{q_{k}}{\cos^{2} \theta_{k}}-\ln\Big(\mbox{det}(B_{k})\Big)-\ln(m_{k})+\ln(q_{k})-\ln~(1-\lambda)^3 \nonumber \\
&=\psi(B_{k}) + \Big( (1- \lambda) M_{k} - \ln (m_{k})- \ln (1-\lambda)^3 \Big) + \Big[ 1- \frac{q_{k}}{\cos^{2} \theta_{k}} + \ln \Big( \frac{q_{k}}{\cos^{2}\theta_{k}} \Big)\Big] \nonumber \\
&~~~~~~~~~~~~~~~~~~+ \ln (\cos^{2} \theta_{k}).
\end{align}
Since the real valued function $h(t)=1-t-\ln(t)$ is always non positive for all $t>0$,  so $\Big[ 1- \frac{q_{k}}{\cos^{2} \theta_{k}} + \ln \Big( \frac{q_{k}}{\cos^{2}\theta_{k}} \Big)\Big]$ is non positive. From (\ref{mmk})  \\
\begin{equation}\label{psi}
0< \psi(B_{k+1}) < \psi(B_{0})+c(k+1)+ \sum_{j=0}^{k} \ln \Big( \cos^{2} \theta_{j} \Big),
\end{equation}
where $c= (1- \lambda) M_{k} - \ln (m_{k})- \ln (1-\lambda)^3 $ can be assumed to be positive with out loss of generality. \\ \\
Assumption 1 implies that $ \nabla f$ Lipschitz continuous. Hence there exists a constant $L_{1} > 0$ such that  $ \| \nabla f(x) - \nabla f(y) \| \leq L_{1} \| x- y \|$ for all $x, y\in int(\mathcal{L})$.
\\\\
 Here we employ Wolfe inexact line search to determine the step length $\alpha_{\mathbf{\bar{k}}}$ at the iterating point $x_k$ for which
\begin{align*}
&f(x_{k}+\alpha_{\mathbf{\bar{k}}}s_{k}) \leq f(x_{k}) + c_{1} \alpha_{\mathbf{\bar{k}}} \nabla f_{k}^{T} s_{k},
\nonumber\\
&\nabla f(x_{k} + \alpha_{\mathbf{\bar{k}}} s_{k})^{T} s_{k} \geq c_{2} \nabla f_{k}^{T} s_{k},
\end{align*}
with $ 0 < c_{1} < c_{2} <1 $. From second inequality
\begin{equation} \label{CCC1}
(\nabla f_{\mathbf{\bar{k}}} - \nabla f_{k})^{T} s_{k} \geq (c_{2}-1) \nabla f_{k}^{T} s_{k},
\end{equation}
while the Lipschitz continuity of $\nabla f$ implies that
\begin{align} \label{CCC2}
(\nabla f_{\mathbf{\bar{k}}} - \nabla f_{k})^{T} s_{k} & \leq \| \nabla f_{\mathbf{\bar{k}}} - \nabla f_{k} \| \|s_{k} \|  \nonumber \\
& \leq L_{1} \|\alpha_{\mathbf{\bar{k}}} s_{k} \| \| s_{k} \|
 = \alpha_{\mathbf{\bar{k}}} L_{1} \| s_{k} \|^{2}.
\end{align}
Combining (\ref{CCC1}) and (\ref{CCC2}), we obtain
$
\alpha_{\mathbf{\bar{k}}} \geq \frac{c_{2} -1}{L_{1}} \frac{\nabla f_{k}^{T} s_{k}}{ \| s_{k} \|^{2}}.
$ Substituting this value of $\alpha_{\mathbf{\bar{k}}}$ in the first inequality of Wolfe's condition we can get
\begin{equation*} 
f_{\mathbf{\bar{k}}} \leq f_{k} - c_{1} \frac{1- c_{2}}{L_{1}} \frac{ (\nabla f_{k}^{T} s_{k})^{2}}{ \| s_{k} \|^{2} } \leq f_{k} - \underline{c} \cos^{2} \theta_{k} \| \nabla f_{k} \|^{2},
\end{equation*}
where $\underline{c}= c_{1} \frac{1-c_{2}}{L_{1}}.$
For each $k$, Assumption 2 implies $f_{k+1} \leq f_{\mathbf{\bar{k}}}$, which gives
$f_{k+1}\leq f_{k} - \underline{c} \cos^{2} \theta_{k} \| \nabla f_{k} \|^{2}$. \\
 Repeated application of this inequality provides $f_{k+1}\leq f_{0} -\underline{c}\sum_{j=0}^{k} \cos^{2} \theta_{j} \| \nabla f_{j} \|^{2}.$
That is, $$\sum_{j=0}^{k} \cos^{2} \theta_{j} \| \nabla f_{j} \|^{2} \leq \frac{1}{\underline{c}}(f_{0}-f_{k+1}).$$
Since $f$ is bounded below, $f_{0}-f_{k+1}$ is less than some positive constant, for all $k$. Taking the limit $k\rightarrow \infty$ we can get
$
\sum_{k=0}^{\infty} \cos^{2} \theta_{k} \| \nabla f \|^{2} < \infty.
$\\\\
Hence  $\cos^{2} \theta_{k} \| \nabla f \|^{2} \rightarrow 0$. Then  $\| \nabla f \| \rightarrow 0$ if  $\cos \theta_{k} > \frac{1}{M'}$, for some positive number $M'$. Then $\cos \theta_{j} \rightarrow 0$ is not true. Otherwise if
 $\cos \theta_{j} \rightarrow 0$ then there exists a natural number $k_{2} >0 $ such that for all $j>k_{2}$ we have $ \ln (\cos^{2} \theta_{j}) < -2c,$
where $c$ is the constant defined in (\ref{psi}).
Then from (\ref{psi}) we get
\begin{equation*}
0< \psi(B_{0}) + c(k+1)+\sum_{j=0}^{k_{1}} \ln ( \cos^{2} \theta_{j})+ \sum_{j=k_{1}+1}^{k} (-2c).
\end{equation*}
R.H.S. of this inequality is negative for large $k$, which is a contradiction.
Therefore there exists a subsequence of indices $\{ j_{k} \}_{k=1,2, \ldots}$ such that $\cos \theta_{jk} \geq \delta >0.$ Hence
$\lim \| \nabla f_{k} \| \rightarrow 0$.
Since the problem is strongly convex, it proves that $x_{k} \rightarrow x^{*}$.
\end{proof}
\begin{theorem} \label{Theorem_super}
Suppose the hessian $G(x)$ is Lipschitz continuous. Then the sequence $\{ x_k \}$, generated by  (\ref{sch}) converges to minimizer of $(P)$
 at super linear order.
\end{theorem}
\begin{proof}
Suppose the sequence $\{x_k\}$, generated by (\ref{sch}) converges to $x^{*}$, which is a solution of $P$.  Denote
 \begin{equation*}
G_{*}=G(x^{*}),~~\tilde{s_{k}}=G_{*}^{\frac{1}{2}}s_{k},~~~ \tilde{y}_{k}=G_{*}^{\frac{1}{2}}y_{k},~~~\tilde{B_{k}}=G_{*}^{-\frac{1}{2}}B_{k} G_{*}^{-\frac{1}{2}}, \nonumber
\end{equation*}
If $\tilde{\theta}_{k}$ is the angle between $\tilde{s}_{k}$ and $\tilde{B}_{k}\tilde{s}_{k}$ then $\cos \tilde{\theta}_{k} = \frac{\tilde{s}_{k}\tilde{B}_{k}\tilde{s}_{k}} {\| \tilde{s}_{k} \| \| \tilde{B}_{k} s_{k}\|}$. Denote $$
 \tilde{q}_{k}=\frac{\tilde{s}_{k}\tilde{B_{k}}\tilde{s}_{k}}{2\| \tilde{s}_{k} \|^{2}},~~ \tilde{M_{k}} = \frac{\| \tilde{y_{k}} \|^{2}}{\tilde{y_{k}}^{T} \tilde{s_{k}}}~~
\mbox{and}~~ \tilde{m}_{k}=\frac{\tilde{y}_{k}^{T} \tilde{s}_{k}}{\| \tilde{s}_{k} \|^{2}} $$
By pre and post multiplying two-phase quasi-Newton update (\ref{compact_form}) by $G_{*}^{-\frac{1}{2}}$ and grouping we get
$
\tilde{B}_{k+1}=\tilde{B}_{k} - (1-\lambda) \frac{\tilde{B}_{k} \tilde{s}_{k} \tilde{s}^{T}_{k} \tilde{B}_{k}}{\tilde{s}_{k}^{T}\tilde{B}_{k} \tilde{s}_{k}}+ (1-\lambda) \frac{\tilde{y}_{k} \tilde{y}_{k}^{T}}{\tilde{y}_{k}^{T} \tilde{s}_{k}}.
$
Following the similar logic of (\ref{psi1}) from Subsection (\ref{Convergence}), we get
\begin{equation*}
\psi(\tilde{B}_{k+1})< \psi(\tilde{B}_{k})+ \Big((1-\lambda) \tilde{M}_k- \ln (\tilde{m}_{k})- \ln~ (1-\lambda)^3 \Big)+\Big[1- \frac{\tilde{q} _{k}}{\cos^{2} \tilde{\theta}_{k}}+ \ln \Big( \frac{\tilde{q} _{k}}{\cos^{2} \tilde{\theta}_{k}} \Big) \Big].
\end{equation*}
From (\ref{yGbar}), $y_{k}= \bar{G}_{k}s_{k}$, so $y_{k}-G_{*}s_{k}=(\bar{G}_{k}-G_{*})s_{k}$ and
$
\tilde{y}_{k}-\tilde{s}_{k}=G_{*}^{-\frac{1}{2}}(\bar{G}_{k}-G_{*})G_{*}^{-\frac{1}{2}}\tilde{s}_{k}.
$
Using the Lipschitz continuity property of hessian matrix, we have
\begin{equation*}
\| \tilde{y}_{k}-\tilde{s}_{k} \| \leq \| G_{*}^{-\frac{1}{2}} \| \| \tilde{s}_{k} \| \| G_{k}- G_{*} \| \leq \| G_{*}^{-\frac{1}{2}} \| \| \tilde{s}_{k} \| L \epsilon_{k},
\end{equation*}
where $\epsilon_{k}=\epsilon_{k}=\mbox{Max} \{ \| x_{\bar{k}} - x^{*} \|, \| x_{k} - x^{*} \|  \}$ and $L$ is Lipschitz constant.
From above inequality,
\begin{equation}\label{Superlinear_theorem}
\frac{\| \tilde{y}_{k} -\tilde{s}_{k} \|}{\| \tilde{s}_{k} \|} \leq \bar{c} \epsilon_{k}, ~~ \mbox{where~$\bar{c}=\| G_{*}^{-\frac{1}{2}} \| L $~ is a real constant}.
\end{equation}
Replacing $s_{k}$ by $p_{\mathbf{\bar{k}}}$ in Theorem 6.6 (Chapter 6 \cite{nocedal2006numerical}) and using (\ref{Superlinear_theorem}) we conclude that
$\lim \frac{\| (B_{k} -G_{*})p_{\mathbf{\bar{k}}} \|}{\| p_{\mathbf{\bar{k}}} \|}=0.$
This is the condition for super liner convergence of the sequence in Theorem \ref{convergence_criterion}, with unit step length $\alpha_{k}=1$ in the vicinity of the solution. Hence the result.
\end{proof}
\subsection{Algorithm} \label{Algo}
In the previous section a two-phase quasi-Newton sequence (\ref{sch}) is developed, $B_{k}$ is updated  at each iteration using intermediate update $B_{\mathbf{\bar{k}}}$. Since solving matrix system has expensive computational cost, so we frame the algorithm for the proposed scheme such that it includes matrix-vector multiplication only. Denote the inverse $B_k$ at $x_{k}$ as $H_{k}$.
As per the proposed scheme $B_{k+1}=\lambda B_{k}+(1- \lambda) B_{\mathbf{\bar{k}}},$
which implies,
\begin{equation}\label{H}
H_{k+1}=B_{k+1}^{-1}= \Bigg( \lambda B_{k}+(1- \lambda) B_{\mathbf{\bar{k}}} \Bigg)^{-1}=\Bigg( \lambda H_{k}^{-1} + (1-\lambda) H_{\mathbf{\bar{k}}}^{-1} \Bigg)^{-1}.
\end{equation}
To frame the algorithm we use the standard inverse Hessian update as follows.
\begin{equation} \label{BFGS_H}
H_{\mathbf{\bar{k}}}= (I- \rho_{k} s_{k} y_{k}^{T}) H_{k} (I - \rho_{k} y_{k} s_{k}^{T})+ \rho_{k} s_{k} s_{k}^{T},~~\mbox{where} ~\rho_{k}= \frac{1}{y_{k}^{T}s_{k}}.
\end{equation}
\begin{algorithm}[H]
 \KwIn{$x_{0}$, $\epsilon>0$}
 $H_{0} \leftarrow I$ ; \\
 $k \leftarrow 0$ ; \\
\hspace{5mm} while ($\| \nabla f_{k} \| > \epsilon $) \\
\hspace{10mm}  $p_{k} \leftarrow -H_{k} \nabla f_{k}$ ; \\
 \hspace{10mm} $x_{\bar{k}} \leftarrow x_{k} + \alpha_{\mathbf{\bar{k}}} p_{\mathbf{\bar{k}}}$ ;  $\alpha_{\mathbf{\bar{k}}}$ is the step-length by Wolfe condition.\\
 \hspace{10mm} Compute $H_{\bar{k}}$ by (\ref{BFGS_H}) ; \\ 
 \hspace{10mm} Compute $H_{k+1}$ by (\ref{H}) ;  \\ 
 \hspace{10mm} $p_{k} \leftarrow -H_{k+1} \nabla f_{k}$ ; \\
 \hspace{10mm} $x_{k+1} \leftarrow x_{k} + \alpha_{k} p_{k}$ ; $\alpha_{k}$ is the step-length by Wolfe condition.\\
 \hspace{10mm} $k \leftarrow k+1$ ; \\
\hspace{5mm} end \\
 \KwOut{$x_{k+1}$}
\caption{{\bf Two Phase quasi-Newton Algorithm} \label{Algorithm1}}
\end{algorithm}

\section{Numerical results and discussions}\label{numerical}
We fix the set up for numerical computation on MATLAB R2015a platform with 4 GB RAM in a 64 bit computer. We fix $\lambda = 0.5$ for proposed two-phase quasi-Newton scheme. The tolerance limit is taken to be $10^{-6}$. The Armijo parameter, Wolfe parameter, backtracking factors are taken to be $10^{-4}$, $0.9$ and $0.5$ respectively. We choose 30 test functions of 10 dimensions from \cite{andrei2008unconstrained} and run the traditional BFGS and the proposed algorithm (Two-phase QN) in this platform with the initial points specified in \cite{andrei2008unconstrained}. For each function, both the algorithms are executed for 5 runs. Note the number of iterations and average execution time in Table 1. It presents that two-phase quasi-Newton scheme takes less number of iterations to reach the solution point whereas execution time remains almost same.\\\\
Furthermore, we provide a graphical illustration (Figure 1) of performance profile for number of iteration and average execution time for both of the algorithms on this test set. The performance ratio defined by Dolan and More \cite{dolan2002benchmarking} is $\rho_{(p,s)}=\frac{r_{(p,s)}}{ \min \{r_{(p,s)}: 1 \leq r \leq n_s\}}$, where $r_{(p,s)}$ refers to the iteration (or, execution time) for solver $s$ spent on problem $p$ and $n_s$ refers to the number of problems in the model test set. In order to obtain an overall assessment of a solver on the given model test set, the cumulative distribution function $P_s(\tau)$ is $P_s(\tau) = \frac{1}{n_p} size\{ p \in \mathscr{T}: \rho_{(p,s)} \leq \tau\},$ where $P_s (\tau)$ is the probability that a performance ratio $\rho_{(p,s)}$ is within a factor of $\tau$ of the best possible ratio.

\vspace{-5mm}
\begin{center}
\begin{table}[H]
\caption{Comparison between BFGS and Two-phase QN}
\vspace{2mm}
\begin{tabular}{ | r | l | r | r | r | r | }
\hline
	  & & \multicolumn{2}{c|} {BFGS} & \multicolumn{2}{c|}  {Two-Phase QN} \\ \hline
	 Sl  & Function   & Iteration & Time(s) & Iteration & Time(s) \\ \hline
	1 & Almost Perturbed Quadratic & 18 & 20.084107& 15 & 15.409993 \\
	2 & ARWHEAD  & 7 & 10.222042 & 7 & 6.131260 \\
	3 & BIGGSB1 & 12 & 4.871641 & 11 & 6.353324\\
	4 & Diagonal 1 & 17 & 25.438489& 13 & 23.955075\\
	5 & Diagonal 2 & 23 & 11.601469 & 22 & 15.697982 \\
	6 & Diagonal 3 & 19 & 14.133412 & 14 & 26.880991\\
	7 & Diagonal 7 & 5 & 7.324530& 7 & 7.011224 \\
	8 & Diagonal 9 & 22 & 16.894874 & 14 & 23.560161 \\
	9 & DIXMANNA DIXMAANL & 12 & 18.307508 & 11 & 18.078639 \\
	10 & DQDRTIC & 13 & 27.743780 & 20 & 26.610668 \\
	11 & EDENSCH & 23 & 29.752709& 18 & 29.497861 \\
	12 & ENGVAL1  & 30 & 29.388006& 25 & 51.026493 \\
	13 & Extended Beale & 22 & 17.23462331 & 20 & 34.315333 \\
	14 & Extended DENSCHNB & 7 & 8.237664028& 7 & 16.228143\\
	15 & Extended Freudenstein and Roth & 10 & 8.573212 & 9 & 7.694981 \\
	16 & Extended PSC1 & 13 & 11.365949 & 12 & 10.933492 \\
	17 & Extended Tridiagonal 1 & 21 & 15.855021 & 20 & 12.977472 \\
	18 & Extended Tridiagonal 2 & 11 & 6.7415813 & 9 & 9.071257\\
	19 & Fletcher & 28 & 35.639960 & 25 & 47.106851 \\
	20 & Generalized PSC1 & 23 & 27.241052 & 14 & 27.129654\\
	21 & Hager & 17 & 17.911313 & 8 & 15.208549 \\
	22 & HIMMELH & 7 & 10.81266821 & 6 & 5.727413 \\
	23 & Partial Perturbed Quadratic & 16 & 18.650512 & 16 & 16.474436 \\
	24 & Perturbed Quadratic Diagonal & 10 & 11.815186 & 5 & 7.448544 \\
	25 & Perturbed Tridiagonal Quadratic & 18 & 18.504252 & 14 & 22.422385 \\
	26 & Quadratic QF1 & 16 & 27.166675 & 11 & 15.076536 \\
	27 & Quadratic QF2  & 23 & 36.647364 & 17 & 28.013425 \\
	28 & Raydan1  & 18 & 14.126658 & 16 & 19.135894 \\
	29 & Raydan2 & 7 & 8.412570 & 5 & 8.376115\\
	30 & Tridia & 15 & 16.038225 & 16 & 18.414270 \\
	   \hline
\end{tabular}
\end{table}
\end{center}


\begin{figure}[H]
\begin{minipage}{.5\textwidth}
  \includegraphics[width=1.2\linewidth]{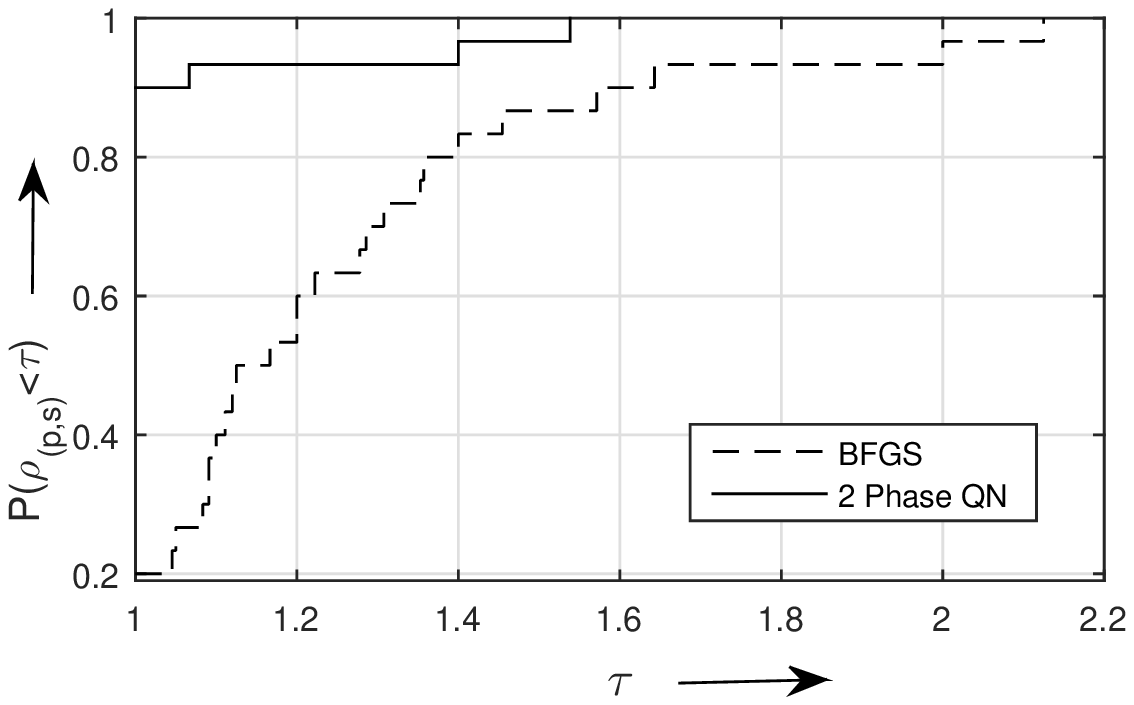}
  \captionof{figure}{Performance profile for number of iterations}
  \label{fig:test1}
\end{minipage}%
\begin{minipage}{.5\textwidth}
  \includegraphics[width=1.2\linewidth]{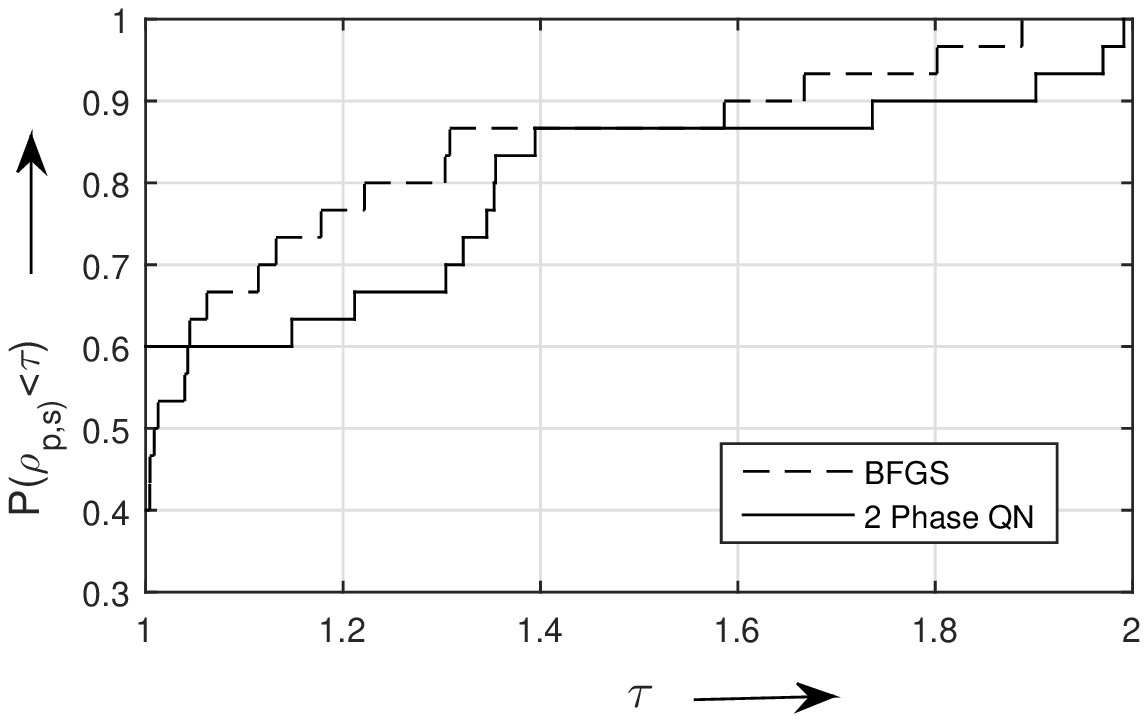}
  \captionof{figure}{Performance profile for average Time}
  \label{fig:test2}
\end{minipage}
\end{figure}

\section{Conclusion}
In this paper we have proposed a two-phase quasi-Newton scheme for optimization problem, which shows super linear convergence property under some suitable assumptions. The algorithm is executed for some test functions and the results show the advantage of the proposed scheme in terms of iterations in comparable time over the traditional schemes, which is provided in the performance profiles.

\end{document}